\documentclass[a4paper,11pt]{amsart}

\usepackage{amssymb,amscd,amsmath}
\usepackage[mathcal]{eucal}
\usepackage{array,float}

\usepackage{xy}
\input xy
\xyoption{all}
\usepackage{pdflscape}
\usepackage{hyperref}
\usepackage{pdflscape}
\usepackage{enumerate}
\usepackage{tikz}
\usepackage{tikz-cd}
\usepackage{relsize}%
\usetikzlibrary{decorations.pathreplacing}
\usetikzlibrary{arrows.meta, positioning}
\usepackage[left=3cm, right=3cm]{geometry}
\usepackage{enumerate}
\numberwithin{equation}{section}
\numberwithin{equation}{subsection}

\newtheorem{theorem}{Theorem}[section]

\newtheorem{lemma}[theorem]{Lemma}

\newtheorem*{remark*}{Remark}

\title[On the CLASSIFICATION OF FINITE GROUPS]{Finite Groups with nearly half as many cyclic subgroups as elements} 

\author{Vaibhav Chhajer}

\address{School of Mathematical Sciences, National Institute of Science Education and Research, An OCC of Homi Bhabha National Institute, Bhubaneswar 752050, Odisha, India}

\email{vchhajer@niser.ac.in}

\author{Sumana Hatui}
\address{School of Mathematical Sciences, National Institute of Science Education and Research, An OCC of Homi Bhabha National Institute, Bhubaneswar 752050, Odisha, India}

\email{sumanahatui@niser.ac.in}

\author{Palash Sharma}
\address{Department of Mathematical Sciences, Indian Institute of Science Education and Research Mohali, Knowledge City, Sector 81, Mohali 140 306, Punjab, India}

\email{ms22001@iisermohali.ac.in}

\begin{document}
\date{}
\begin{abstract}
Suppose $C(G)$ denotes the set of all cyclic subgroups of a finite group $G$, and $\mathcal{O}_{2}(G)$ denotes the number of elements of order $2$ in $G$.
 In \cite{MT}, an open problem was asked to classify the groups $G$ with  $|C(G)|=|G|-r$, where $2 \leq r \leq |G|-1$. 
In this article, first  we
show that, for an odd prime $p$, there are infinitely many groups $G$ with $|C(G)|= \frac{|G|}{2}$,  $|C(G)|=\frac{|G|}{p^{q-1}}$ (for prime $q\neq p)$,  or $|C(G)|=\frac{|G|}{2}+2^{k},  k\geq 0$.
Then, we  partially answer the open question by classifying finite groups $G$ having 
$\frac{|G|}{2}-1\leq |C(G)| \leq \frac{|G|}{2}+1$ for some fix values of  $\mathcal{O}_{2}(G)$.
Finally, we provide a complete list of finite groups $G$ having $|C(G)|=\frac{|G|+(2r+1)}{2}$ for $r\geq-1$.
\end{abstract}
\maketitle



 

\section{Introduction} 
Throughout this paper, $\mathbb{Z}_{n}$, $Q_{2^{k}}$ and $M_{2^k}$ denote the cyclic group of order $n$,  the generalized quaternion group of order $2^{k}$, and  the modular maximal-cyclic group of order $2^{k}$ respectively. By SmallGroup$[i,j]$, we denote  the $j$-th group of order $i$ in GAP’s SmallGroups library. Suppose $|G:H|$ denotes the index of a subgroup $H$ in $G$, $\phi$ denotes the Euler totient function, and $\exp(G)$ denotes the exponent of a group $G$. Let $|G|$ and $|x|$ denote the order of a group $G$ and the order of an element $x \in G$, respectively.


It is always an interesting problem to classify finite groups $G$ by their certain properties. The classification of $G$ by the order of $C(G)$  received attention recently. It is easy to see that $|C(G)|= |G|$ if and only if $G$ is an elementary abelian $2$-group.
In $2015$, Tărnăuceanu \cite{MT} classified the groups with $|C(G)|=|G|-1$, imposing the following open problem.\\

\noindent \textbf{Open problem:} Describe the finite groups $G$ with $|C(G)| = |G|-r,$ where $2 \leq r \leq |G|-1$.\\

In $2016$, the answer was given for $r=2$ in \cite{MT1}. As a continuation, all finite groups with $|G|-5\leq |C(G)|\leq |G|-3$ and $2\leq|C(G)|\leq 13$ have been classified in \cite{AA,AA1,RJL,  LCQ,SR,SR1}. 
Also, the groups $G$ for certain values of the function $\alpha(G)=\frac{|C(G)|}{|G|}$ have been  studied by several authors in \cite{GS,GL,ML}, listing a few. 


Our first main result is the following. 
\begin{theorem}\label{mainresult}
Let $p$ be an odd prime. Then, there are infinitely many groups $G$ having $|C(G)|= \frac{|G|}{2}$,  $|C(G)|=\frac{|G|}{p^{q-1}},$ for prime $q\neq p$,  or $|C(G)|=\frac{|G|}{2}+2^{k},  k\geq 0$. 
 \end{theorem}

%
%

Since there are infinitely many groups for $|C(G)|=\frac{|G|}{2}$, it is challenging  to classify the groups for this value of $|C(G)|$. Hence, in the following result, we fix the value of $\mathcal{O}_{2}(G)$ and classify the groups $G$ having $|C(G)|=\frac{|G|}{2}$. Note that, as $|G|$ is even, $\mathcal{O}_{2}(G)$ is always an odd number.

\begin{theorem}\label{mainresult_2}
Suppose $G$ is a finite group with $|C(G)|=\frac{|G|}{2}$. Then we have the followings. 
\begin{enumerate}
      \item $\mathcal{O}_{2}(G)=1$ if and only if $G\cong \mathbb Z_8, \mathbb Z_{12}, Q_{16}$ or $ \mathbb{Z}_{3}\rtimes Q_8$.
        
\item $\mathcal{O}_{2}(G)=3$ if and only if $G\cong \mathbb Z_8 \times \mathbb Z_2, M_{16}$, SmallGroup$[32,10]:$ $ Q_8 \rtimes \mathbb Z_4$, SmallGroup$[32,13]:$ $\mathbb Z_8 \rtimes \mathbb Z_4 $, SmallGroup$[32,14]:$ $\mathbb Z_8 \rtimes \mathbb Z_4$, $Q_{16}\times \mathbb{Z}_2, \mathbb Z_{12}\times \mathbb Z_2$,$ \mathbb Z_4\times (\mathbb{Z}_{3} \rtimes \mathbb{Z}_{4})$, SmallGroup$[48,12]:$ $(\mathbb{Z}_{3} \rtimes \mathbb{Z}_{4})\rtimes \mathbb{Z}_4$, SmallGroup$[48,13]:$ $\mathbb{Z}_{12}\rtimes \mathbb{Z}_4$ or $\mathbb{Z}_2\times (\mathbb{Z}_{3}\rtimes Q_{8})$.

        \item There is no group with $\mathcal{O}_{2}(G)=\frac{|G|}{2}-1$ or $\mathcal{O}_{2}(G)=\frac{|G|}{2}-2$.

                \item $\mathcal{O}_{2}(G)=\frac{|G|}{2}-3$ if and only if $G\cong \mathbb Z_8$.

\end{enumerate}
\end{theorem}
In the following results, we classify the groups $G$ having $|C(G)|=\frac{|G|}{2}-1$ or $ \frac{|G|}{2}+1$ for some fix values of $\mathcal{O}_{2}(G)$.
 \begin{theorem}\label{mainresult_3}
Suppose $G$ is a finite group with $|C(G)|=\frac{|G|}{2}-1$. Then we have the followings.   
\begin{enumerate} 
        \item $\mathcal{O}_{2}(G)=1$ if and only if $G\cong \mathbb Z_{10} $ or SmallGroup$[20,1]:$ $\mathbb{Z}_{5}\rtimes \mathbb{Z}_{4}$.

        \item There is no group with $\mathcal{O}_{2}(G)=\frac{|G|}{2}-2$ and $\mathcal{O}_{2}(G)=\frac{|G|}{2}-3$.
        
        \item $\mathcal{O}_{2}(G)=\frac{|G|}{2}-4$ if and only if $G\cong \mathbb Z_{10}$.

    \end{enumerate}
\end{theorem}

  \begin{theorem}\label{mainresult_4}
Suppose $G$ is a finite group with $|C(G)|=\frac{|G|}{2}+1$. Then we have the followings.  
\begin{enumerate} 
\item There are infinitely many groups with $\mathcal{O}_{2}(G)=1$.
\item  There is no group with $\mathcal{O}_{2}(G)=3.$
    \item $\mathcal{O}_{2}(G)=\frac{|G|}{2}$ if and only if $G\cong \mathbb{Z}_{2}$.
    \item $\mathcal{O}_{2}(G)=\frac{|G|}{2}-1$ if and only if $G\cong \mathbb{Z}_{4}$.
    \item $\mathcal{O}_{2}(G)=\frac{|G|}{2}-2$ if and only if $G\cong \mathbb{Z}_{6}$.
\end{enumerate}
\end{theorem}

\begin{theorem}\label{mainresult_5}
 Let $G$ be a group such that $|C(G)|=\frac{|G|+(2r+1)}{2}, r\geq -1$ is an integer. 
 Then \begin{enumerate}
        \item $r=-1$ if and only if $G\cong \mathbb{Z}_5$.
        
\item $r=0$ if and only if $G\cong\{e\}$ or $G$ is a group with $\exp(G)=3$.
\item For $r\geq 1$, there is no group.
    \end{enumerate}

 \end{theorem}

 The next result gives the bound on the number of prime divisors of the order of group $G$ when  $|C(G)|=\frac{|G|}{2}+m$. 
\begin{lemma}\label{bound}
Let $G$ be a group with $|C(G)|=\frac{|G|}{2}+m$, $m$ is an integer. If $p$ is a prime divisor of $|G|$, then $p\leq \mathcal{O}_2(G)+4-2m$.
\end{lemma}
It follows from the above result that, if $G$ is a group with $|C(G)|=\frac{|G|}{2}+m$, then $\mathcal{O}_2(G)>2m-3$.
Since there are infinitely many groups $G$ with $|C(G)|=\frac{|G|}{2}, \frac{|G|}{2}+2^k$ for $k \geq 0$ (by Theorem \ref{mainresult}), it is natural to ask to classify the groups $G$ by fixing $\mathcal{O}_{2}(G)$. Hence, we state the following open question.\\

\noindent \textbf{Open problem:} Classify the finite groups $G$ with $|C(G)|=\frac{|G|}{2}+ m$, for a fixed integer $m$, for the values of $\mathcal{O}_{2}(G)$, where $2m-3<\mathcal{O}_{2}(G)< |C(G)|$.

\section{prerequisites}

 First, we recall the following basic result of group theory which will be used repeatedly in this paper.
\begin{lemma}\label{normal}
    Let $G$ be a finite group and $H\leq G$, if $H$ is a unique cyclic subgroup of order $|H|$, then $H\trianglelefteq G$. 
\end{lemma}

Define a relation $\sim$ on $G$ by the following:  $a \sim b$ if $a$ and $b$ are the generators of the same cyclic subgroup of $G$. Then, $\sim$ is an equivalence relation on $G$. Let $cl(x)$ denote the equivalence class containing $x$. Now, we describe a key step, the first stair towards the proof of the main results.
\subsection{\textbf{Key step:}}\label{keystep} Let $\mathcal{O}_{2}(G)=r \neq 0$. 
Suppose $|G|=2n$ and $\{x_i\}_{i=1}^{2n}$ be the elements of $G$ such that $x_1 = e$ and $x_i^2 = e$ for $2 \leq i \leq r+1$.
Suppose $|C(G)|=k$. Then $k \geq r+1$, without loss of generality, let $x_1, x_2, \cdots,x_k$ be the generators of those $k$ cyclic subgroups of $G$.

We have $cl(x_1) = \{e\}$ and $cl(x_i) = \{x_i\}$ for $2 \leq i \leq r+1$.
Thus $|cl(x_j)| \geq 2~ $ for  $~r+2\leq j \leq k $ and so $cl(x_j)$ must contain at least one element $x_t$ for some $t\in \{k+1, \cdots, 2n\}$. 
First, we distribute exactly one element from $\{x_{k+1}, \cdots, x_{2n}\}$ to each of $cl(x_j)$, then we are left with $s=2(n-k)+r+1$ many elements which are not assigned to any class yet.

\begin{center}
\begin{tikzpicture}[every node/.style={anchor=base}]
    \node (e=x1) at (-0.3,0) {$e=x_1$};
    \node at (2.7,0.6) {\small \text{order $2$ elements}};
    \draw[decorate,decoration={brace,amplitude=5pt}] (0.8,0.3) -- (4.4,0.3);
    \node (x2) at (1,0) {$x_2$};
    \node (x3) at (2,0) {$x_3$};
    \node (dots1) at (3,0) {$\cdots$};
    \node (x1) at (4,0) {$x_{r+1}$};
    \node (xkplus2) at (5.1,0) {$x_{r+2}$};
    \node (dots2) at (6.1,0) {$\cdots$};
    \node (x2k) at (7.1,0) {$x_{k}$};
    \node (dots2) at (6.1,-0.6) {$\cdots$};
    \node (xkplus1) at (5.1,-1.2) {$x_{k+1}$};
     \node (dots2) at (6.1,-1.2) {$\cdots$};
    \node (x2kminuskplus1) at (7.2,-1.2) {$x_{2k - (r+1)}$};
    \draw[<->] (5.1,-0.2) -- (5.1,-1.0);
    \draw[<->] (7.1,-0.2) -- (7.1,-1.0);
    \node (x2kmins) at (8.8,-1.2) {$x_{2k-r}$};
    \node (dots3) at (9.8,-1.2) {$\cdots$};
    \node (x2n) at (10.8,-1.2) {$x_{2n}$};
    \node at (9.6,-1.9) {\small \text{ $s$ many elements}};
    \draw[decorate,decoration={brace,amplitude=5pt}] (11.2,-1.4) -- (8.2,-1.4);
\end{tikzpicture}
\end{center}
Now assigning the remaining $s$ elements to the equivalence classes  is equivalent to considering all the partitions of $``s"$. We can ignore those partitions which contain at least one odd number as in those cases, the no. of generators of the corresponding cyclic subgroup will be odd, a contradiction to the fact that $\phi(k)$ is even,  $\forall k\in \mathbb{N}_{\geq 3}$. 
So we consider even partitions of $s$ and assign those $s$ elements in the equivalence classes with each possibilities.
\\

\begin{theorem}\label{|G|/p}
Let $p$ be an odd prime  such that $p \nmid |G|$ or $p=2$.  Then $|C(G\times \mathbb{Z}_{p})|=2|C(G)|$.
\end{theorem}
\begin{proof}
   Recalling that the number of cyclic subgroups of a finite group $G$ can be computed as
follows  $$|C(G)|=\sum_{x\in G}\frac{1}{\phi(|x|)}.$$ Using it for $\widetilde{G}=G\times\mathbb{Z}_{p}$, we get 
\begin{equation*}
    \begin{aligned}
        |C(\widetilde{G})|&=\sum_{(x,y)\in \widetilde{G}} \frac{1}{\phi(lcm(|x|,|y|))}\\
        &= \sum_{y=0} \frac{1}{\phi(|x|)}+\sum_{\substack{y\neq 0\\ p \mathlarger{\nmid} |x|}} \frac{1}{\phi(p|x|)}+ \sum_{\substack{y\neq 0\\ p \mathlarger{\mid} |x|}} \frac{1}{\phi(|x|)}\\ &=2|C(G)|+(p-2)\sum_{\substack{x\in G \\ p \mathlarger{\mid} |x|}} \frac{1}{\phi(|x|)}.
    \end{aligned}
\end{equation*}
 When $p=2$ or $p\nmid |G|$, we have $|C(G\times \mathbb{Z}_{p})|=2|C(G)|$. 
\end{proof}

\section{Proof of main results}
In this section we prove our main results.\\

\noindent \textbf{Proof of Theorem \ref{mainresult}.}
\begin{proof}
Let $G$ be any $2$-group with $|C(G)|=\frac{|G|}{2}$.  
Then for any prime $p$, it follows from  Theorem \ref{|G|/p},
$$|C(G\times\mathbb{Z}_{p})|=2|C(G)|=|G|=\frac{|G\times\mathbb{Z}_{p}|}{p}.$$ 
 Using GAP \cite{GAP}, $M_{16}$ is a solution for $|C(G)|=\frac{|G|}{2}$. Hence, it follows that $M_{16}\times (\mathbb Z_{2})^n$ are also solutions for $|C(G)|=\frac{|G|}{2}$ for any $n\in \mathbb N$. Now,  for odd prime $p$, $M_{16}\times (\mathbb Z_{2})^n\times \mathbb Z_p$ are solutions for $|C(G)|=\frac{|G|}{p}$. 
Now, for any prime $q\neq p$, $M_{16}\times (\mathbb Z_{2})^n\times \mathbb Z_q \times \mathbb{Z}_{p^{q-1}}$ are solutions for $|C(G)|=\frac{|G|}{p^{q-1}}$ as $|C(M_{16}\times (\mathbb Z_{2})^n\times \mathbb Z_q \times \mathbb{Z}_{p^{q-1}})|=|C(M_{16}\times (\mathbb Z_{2})^n)|\cdot |C(\mathbb{Z}_q)|\cdot |C(\mathbb{Z}_{p^{q-1}})|$.


Taking $G$ a group with $\exp{(G)}=3$, it is easy to check that $|C(G)|=\frac{|G|+1}{2}$. 
Consider the group $\widetilde{G}=G\times \mathbb{Z}_{2}^{k+1}$ for  $k\geq 0$. Using Theorem \ref{|G|/p} inductively, we get $$|C(\widetilde{G})|=2^{k+1}|C(G)|=2^{k}|G|+2^{k}=\frac{|\widetilde{G}|}{2}+2^{k}.$$
\end{proof}

\subsection{Groups with \texorpdfstring{$|C(G)|=\frac{|G|}{2}$}{|C(G)|=|G|/2} and \texorpdfstring{$\mathcal{O}_2(G)=1$}{O2(G)=1}}\label{o2(G)=1, |G|/2}

 Let $G$ be a group of order $2n$ with $|C(G)| = \frac{|G|}{2} = n$ and $\mathcal{O}_2(G)=1$. For $n = 1, 2, 3$, it can be easily checked that $|C(G)| \neq \frac{|G|}{2}$. So,  assume that $n\geq 4$. 
 Let $\{x_i\}_{i=1}^{2n}$ be the elements of $G$ such that $x_1 = e$ and $x_2^2 = e$. 
Taking $r=1, k=n$, by our key step given in Section (\ref{keystep}), we have $s=2$. So we are left with two elements say $x_\alpha , x_\beta \in \{x_{n+1}, \cdots, x_{2n}\}$ which are not assigned to any class yet. Let, $x_\alpha, x_\beta\in cl(x_i)$ for some $i\in \{3,\cdots,n\}$.  In this case $|cl(x_i)|=4$. Hence,   $G$ must satisfy the following properties:

$(A)$ There is a unique cyclic subgroup (say $H_0$) of order $2$
and a unique cyclic subgroup (say $N$) of order $m$ such that $\phi(m) = 4$. Thus $m\in \{5,8,10,12\}$. 
By Lemma \ref{normal}, $N\trianglelefteq G$.

$(B)$ For $1\leq i \leq {n-3}$, there are $n-3$ cyclic subgroups, say $H_i$, of order $m_i$ such that $\phi(m_i) = 2$. Thus $m_i\in \{3,4,6\}$. 

Let's call conditions $(A)$ and $(B)$ collectively as property $(P)$.
\begin{lemma}\label{lemma1}
   There is no group $G$ satisfying $(P)$ and one of the following conditions.
   \begin{enumerate}
       \item  $N \cong \mathbb{Z}_{10}$.

       \item $N \cong \mathbb{Z}_{5}$. 

        \item $N \cong \mathbb{Z}_{8}$ and $H_i \cong \mathbb{Z}_{3}$ for some $i\in \{1,\cdots,n-3\}$.

        \item $N \cong \mathbb{Z}_8$ and $H_i \cong \mathbb{Z}_{6}$ for some $i\in \{1,\cdots,n-3\}$.

   \end{enumerate}
\end{lemma}
\begin{proof}
\begin{enumerate}
    \item If $\mathbb{Z}_{10} \leq G$, then $\mathbb{Z}_5<\mathbb{Z}_{10}\leq G$, which contradicts $(P)$.
    
    \item If $N\cong \mathbb{Z}_{5}$, then $NH_{0}$ is a subgroup of $G$ isomorphic to $\mathbb{Z}_{10}$, a contradiction to $(P)$.

     \item  Let $H_{i}\cong \mathbb{Z}_{3}$ for some $i \in \{1,\cdots, n-3\}$ then $NH_{i}\cong \mathbb{Z}_{8}\rtimes \mathbb{Z}_{3} \cong \mathbb{Z}_{24}$, a contradiction.
    \item Suppose $H_{i}\cong \mathbb{Z}_{6}$ for some $i \in \{1,\cdots,n-3\} $. Then there is a $k \in \{1,\cdots,n-3\}$, such that $H_k < H_i$ and $H_{k}  \cong \mathbb{Z}_{3}$, which contradicts (3).
\end{enumerate}
\end{proof}

\begin{lemma}\label{Lemma2}
  If $G$ satisfies $(P)$, then the following is true.
    \begin{enumerate}[(i)]
    \item
    If $N\cong \mathbb{Z}_{8}$, then $|G|=2^{k}$ for $k\in\{3,4\}$ and $G \cong \mathbb{Z}_{8} $ or 
    $Q_{16}$.

    \item If $N\cong \mathbb{Z}_{12}$, then $|G|=2^{a}3$ for $a\in \{2,3\}$ and $G\cong \mathbb{Z}_{12}$ or $ \mathbb{Z}_{3}\rtimes Q_8$.
    \end{enumerate}

\end{lemma}
\begin{proof}
(i)    If $N\cong \mathbb{Z}_{8}$, then by Lemma \ref{lemma1} and property $(P)$, $|G|=2^{k}$ for some $k$. Assume that $k\geq 5$ then $G$ has a subgroup $H$ of order $32$. By GAP \cite{GAP},  $H$ has more than one element of order $2$, $H\cong \mathbb{Z}_{32}$ or $H\cong Q_{32}$. As $\mathbb{Z}_{32}$ and $Q_{32}$ contradict $(P)$, there is no such group $G$ of order $2^{k}$ for $k\geq 5$. So $k\in\{3,4\}$. Now using GAP \cite{GAP} we see that the only groups that satisfy $(P)$ are $\mathbb{Z}_{8}$ and $Q_{16}.$\\

(ii) First, we show that if
$H_i\cong \mathbb{Z}_3$ or $H_i\cong \mathbb{Z}_6$,
then $H_i<N$. Suppose $H_i\cong \mathbb{Z}_3$ and $H_i\cap N = \{e\}$, then $NH_i\cong \mathbb{Z}_{12}\rtimes \mathbb{Z}_3 \cong \mathbb{Z}_{36}$, a contradiction. 
    Suppose $H_i\cong \mathbb{Z}_6$, then there is a  $H_{k}\cong \mathbb{Z}_{3}$ such that $H_{k} < H_{i}$. So $H_{k} <N$.
  The result follows since $H_i\cap N$ also contains the unique element $x_2$ of order $2$. Hence, the subgroups of $G$ isomorphic to $\mathbb Z_3$ or $\mathbb Z_6$ are unique.

In this case, we have
$|G|=2^{a}3^b$. We claim that $b=1$, suppose not; say $b\geq 2$, then $G$ admits a subgroup $H$ of order $9$, then $H\cong \mathbb{Z}_{3}\times \mathbb{Z}_{3}$ or $H\cong \mathbb{Z}_{9}$. The former is not possible because there is a unique cyclic subgroup of order $3$ in $G$, and the latter is not possible because it contradicts $(P)$.
  Now, if $a\geq 5$, then $G$ has a subgroup $H$ of order $32$, which contradicts $(P)$ by the same argument used in Lemma \ref{Lemma2} $(i)$. For $1 \leq a\leq 4$, using GAP \cite{GAP}, we have either $G\cong \mathbb{Z}_{12}$ or $\mathbb{Z}_{3}\rtimes Q_8$.
  
\end{proof}

\subsection{Groups with \texorpdfstring{$|C(G)|=\frac{|G|}{2}$}{|C(G)|=|G|/2} and \texorpdfstring{$\mathcal{O}_2(G)=3$}{O2(G)=3}} \label{o2(G)=3, |G|/2} 
For $n \leq 6$, it can be easily checked $|C(G)| \neq \frac{|G|}{2}$ or $\mathcal{O}_{2}(G)\neq 3$ using GAP \cite{GAP}. So, let us assume that $n\geq 7$, and $\{x_i\}_{i=1}^{2n}$ be the elements of $G$ such that $x_1 = e$ and $x_2^2=x_{3}^2=x_{4}^2 = e$. 
Taking $r=3, k=n$, by our key step given in Section \ref{keystep} we have $s=4$. So we are left with four elements, say $x_\alpha , x_\beta,x_{\gamma},x_{\delta} \in \{x_{n+1}, \cdots, x_{2n}\}$, which have not yet been assigned to any class. Thus, we have two cases.\\  

 \textbf{Case 1: $x_{\alpha},x_{\beta},x_{\gamma},x_{\delta} \in cl(x_{i})$} for some $i\in\{5,\cdots,n\}$.

 In this case, $G$ must satisfy the following $3$ conditions.

$(A_1)$ There are three cyclic subgroups of order two and a unique cyclic subgroup (say $N$) of order $m$ such that $\phi(m) = 6 \implies m\in \{7,9,14,18\}$. By Lemma \ref{normal}, $N\trianglelefteq G$.

$(B_1)$ There are $n-5$ cyclic subgroups, say $H_1,\cdots,H_{n-5}$ with orders $m_1,\cdots,m_{n-5}$ such that $\phi(m_i) = 2\implies m_i\in \{3,4,6\} ~\forall ~i\in \{1,\cdots,n-5\}$.

 Let's call conditions $(A_1)$ and $(B_1)$ collectively as property $(P_1)$.
\\

\textbf{Case 2:} $x_{\alpha},x_{\beta}\in cl(x_i)$ and $ x_{\gamma},x_{\delta} \in cl(x_{j})$ for some $i,j\in\{5,\cdots,n\}$ such that $i\neq j$.

  In this case, $G$ must satisfy the following $3$ conditions.

$(A_2)$ There are exactly three cyclic subgroups of order 2,
and exactly two cyclic subgroups (say $K_1$ and $K_2$) of order $m_1$ and $m_2$ respectively, 
such that $\phi(m_i) = 4 \implies m_i\in \{5,8,10,12\}$, for $i=1,2$.

$(B_2)$ There are $n-6$ cyclic subgroups, say $H_1,\cdots,H_{n-6}$ with orders $m_1,\cdots,m_{n-6}$ 
such that $\phi(m_i) = 2\implies m_i\in \{3,4,6\} $ $\forall i\in \{1,\cdots,n-6\}$.

 Let's call conditions $(A_2)$ and $(B_2)$ collectively as property $(P_2)$.

 \begin{lemma}\label{lemma4}
    A group $G$ satisfying $(P_1)$ cannot have a subgroup isomorphic to $\mathbb{Z}_{18},\mathbb{Z}_{14},\mathbb{Z}_{9}$ or $\mathbb{Z}_{7}$.
\end{lemma}
\begin{proof}
     If $\mathbb{Z}_{14} \leq G$ or $\mathbb{Z}_{18} \leq G$, then $\mathbb{Z}_7<\mathbb{Z}_{14}\leq G$ or $\mathbb{Z}_9<\mathbb{Z}_{18}\leq G$ respectively. Both  contradict $(P_1)$. If $N\cong  \mathbb{Z}_{9}$, then for  a cyclic subgroup of order $2$ (say $H_0$), $NH_{0}\cong \mathbb{Z}_{9}\rtimes \mathbb{Z}_{2}\cong D_{18}$ or $\mathbb{Z}_{18}$, contradict $(P_1)$. If $N \cong \mathbb{Z}_{7}$, then $NH_{0}$ is either $D_{14}$ or $\mathbb{Z}_{14}$,  contradict $(P_{1})$.
     
    \end{proof}
    
\begin{lemma}\label{lemma5}
   There is no group $G$ satisfying $(P_2)$ and one of the following conditions.
   \begin{enumerate}
\item $K_1$ or $K_2 \cong \mathbb Z_{10}$.

\item $K_1$ or $K_2 \cong \mathbb Z_{5}$.

        \item $K_1 \cong \mathbb{Z}_{8}$ and $K_2 \cong \mathbb{Z}_{12}$.

   \end{enumerate}
\end{lemma}
\begin{proof} 
$(1)$ If $K_1\cong \mathbb{Z}_{10}$, then $K_2\cong \mathbb{Z}_5$ such that $K_2 < K_1$. Taking an element, say $x_2$, of order $2$ such that $x_2\notin K_1$, we get a subgroup $K_2\langle x_2\rangle $ of order $10$, contradicting $(P_2)$.

    $(2)$ Suppose $K_1\cong \mathbb{Z}_5$. If $K_2\ncong \mathbb{Z}_5$, then $K_1\trianglelefteq G$. Thus, the group will contain either $D_{10}$ or $\mathbb{Z}_{10}$, a contradiction.

    Now assume, $K_2\cong \mathbb{Z}_5$. Then, using the previous argument, we conclude that $K_i$ are not normal in $G$, and they are conjugates of each other. 
    If any element of order $2$ lies inside $N(K_1)$, then as $K_1\trianglelefteq N(K_1)$, we will get a subgroup of order $10$,  contradicting $(P_2)$. Also note that $3 \nmid |N(K_1)|$ as we will have an element of order $15$, contradicting $(P_{2})$. Thus $N(K_1)$ must be a $5$-group. Therefore, $|N(K_1)| = 5^k$. 
    Now, if $k>1$, we will get a subgroup of order $25$ which would be either $\mathbb{Z}_5\times \mathbb{Z}_5$ or $\mathbb{Z}_{25}$, both are not possible. Hence, $|N(K_1)| = 5 \implies |G| = 2|N(K_1)| = 10$, and no such group of order $10$ satisfies $|C(G)| = \frac{|G|}{2}$.


$(3)$ In this case, for some $i\in\{1,\cdots,n-6\}$,  we have $H_i< K_2\leq G$  such that $H_{i} \cong \mathbb{ Z}_3$. Clearly $K_1\trianglelefteq G$ and $K_1 \cap H_{i}= \{e\}$, so $K_1H_i\leq G$ with $K_1H_i\cong \mathbb{Z}_{8}\rtimes \mathbb{Z}_{3}\cong \mathbb{Z}_{24}$, a contradiction to $(P_2)$.
\end{proof} 

\begin{lemma}\label{lemma6}
   If $G$ satisfies $(P_2)$ and $K_{1}, K_{2} \cong \mathbb{Z}_{8}$, then $G$ is a $2$-group. Furthermore $G \cong \mathbb{Z}_8\times \mathbb{Z}_2, M_{16}, $ SmallGroup$[32,10]: $ $ Q_8 \rtimes \mathbb Z_4$, SmallGroup$[32,13]:$ $\mathbb Z_8 \rtimes \mathbb Z_4 $, SmallGroup$[32,14]:$ $\mathbb Z_8 \rtimes \mathbb Z_4$ or $Q_{16}\times \mathbb{Z}_2$.
\end{lemma}
\begin{proof}
  Suppose one of the $K_{i}$ is normal (say $K_1$) in $G$. If $H_i\cong \mathbb{Z}_3$ exists, then $K_1 H_i$ is a subgroup of order $24$,  a contradiction. If $K_{i}$ are not normal, then they must be conjugates of each other. In that case, $|G:N(K_1)| = 2$, and so $N(K_1)\trianglelefteq G$. If $H_i\cong \mathbb{Z}_3$ exists, then it must lie inside $N(K_1)$; otherwise, $|H_i N(K_1)|> |G|$. But $K_1H_{i}\cong \mathbb Z_{24}$,  a contradiction. Hence, $G$ must be a $2$- group. 

 For $|G| = 16$ we have two solutions $\mathbb{Z}_8\times \mathbb{Z}_2$ and $M_{16}$. For $|G| = 32$ we have four solutions which are SmallGroup$[32,10]:$ $ Q_8 \rtimes \mathbb Z_4$, SmallGroup$[32,13]:$ $\mathbb Z_8 \rtimes \mathbb Z_4 $, SmallGroup$[32,14]:$ $\mathbb Z_8 \rtimes \mathbb Z_4$, and $Q_{16}\times \mathbb{Z}_2$.  For $|G| = 64$,  there is no solution by GAP \cite{GAP}. When $|G|\geq 128$, then $G$ admits a subgroup of order $128$, by GAP \cite{GAP} any group of order $128$ contains more than three elements of order $2$, more than two $\mathbb{Z}_8$ are sitting in it or $\mathbb{Z}_{16}$ is sitting in it, giving a contradiction in each case.
\end{proof}
\begin{lemma}\label{lemma7}
Suppose $G$ satisfies $(P_2)$ and $K_{1}, K_{2} \cong \mathbb{Z}_{12}$, then $G$ is a $\{2,3\}$-group. Furthermore $G \cong \mathbb Z_{12}\times \mathbb Z_{2},$ $\mathbb Z_4\times (\mathbb{Z}_{3}\rtimes \mathbb{Z}_{4})$, SmallGroup$[48,12]:$ $(\mathbb{Z}_{3} \rtimes \mathbb{Z}_{4})\rtimes \mathbb{Z}_4$, SmallGroup$[48,13]:$ $\mathbb{Z}_{12}\rtimes \mathbb{Z}_4$ and $\mathbb{Z}_2\times (\mathbb{Z}_{3}\rtimes Q_{8})$.
\end{lemma}
\begin{proof}
   As $K_{1}, K_{2} \cong \mathbb{Z}_{12}$ so $|G|= 2^a 3^b$. Now, if $H_i\leq G$ such that $H_i\cong \mathbb Z_3$, then we claim that $H_i$ must lie inside both $K_1$ and $K_2$ and so it is unique in $G$. On contrary lets assume that $H_i$ lies outside $K_1$ then if $K_1\trianglelefteq G$ we get a subgroup $K_1H_{i}\cong \mathbb{Z}_{12}\rtimes \mathbb{Z}_{3}\cong \mathbb{Z}_{12}\times \mathbb{Z}_{3}$, which contains more than $2$ copies of $\mathbb{Z}_{12}$, a contradiction to $(P_{2})$. If $K_2$ is a conjugate of $K_1$ then $|G:N(K_1)| = 2$ and $H_i$ should lie inside $N(K_1)$ otherwise $|H_i N(K_1)|>|G|$. Since $K_1\trianglelefteq N(K_1)$, we will again get a subgroup $K_1H_{i}\cong \mathbb{Z}_{12}\times \mathbb{Z}_{3}$, contradicting $(P_{2})$. Now we claim that $b=1$, because if $b>1$, then we get a subgroup of order $9$ which would be either $\mathbb{Z}_9$ or $\mathbb{Z}_3\times \mathbb{Z}_3$, the former contradicts $(P_2)$ and the later contradicts the fact that there is a unique $\mathbb{Z}_{3}$ in $G$. Hence, $|G| = 2^a \cdot3$. 
   
   Using GAP \cite{GAP}, for $|G| = 2^3\cdot 3$, $\mathbb Z_{12}\times \mathbb Z_{2}$ is the only solution. For $|G| = 2^4\cdot 3$ we have four solutions namely $\mathbb Z_4\times (\mathbb{Z}_{3} \rtimes \mathbb{Z}_{4})$, SmallGroup$[48,12]:$ $(\mathbb{Z}_{3} \rtimes \mathbb{Z}_{4})\rtimes \mathbb{Z}_4$, SmallGroup$[48,13]:$ $\mathbb{Z}_{12}\rtimes \mathbb{Z}_4$ and $\mathbb{Z}_2\times (\mathbb{Z}_{3}\rtimes Q_{8})$. Now for $a\geq 5$, as $H_i\trianglelefteq G$, so $G$ contains a subgroup of order $2^5\cdot3$  and all groups of order $2^5\cdot 3$ have either more than three elements of order $2$ or more than two $\mathbb{Z}_{12}$ are sitting or $\mathbb{Z}_8$ is sitting inside it contradicting $(P_2)$.
\end{proof}
   
\subsection{Groups with \texorpdfstring{$|C(G)|=\frac{|G|}{2}$}{|C(G)|=|G|/2} and \texorpdfstring{$\frac{|G|}{2}-3 \leq$} \texorpdfstring{$\mathcal{O}_2(G)$} $\leq \frac{|G|}{2}-1$}\label{2.4}

\subsubsection{Groups with \texorpdfstring{$|C(G)|=\frac{|G|}{2}$}{C(G)=|G|/2} and \texorpdfstring{$\mathcal{O}_2(G)=\frac{|G|}{2}-1$}{O2(G)=|G|/2-1}}\label{2.4.1}
 In this case, we have $\frac{|G|}{2}-1$ many cyclic subgroups of order $2$, and one of identity, hence $|C(G)|= |G| \implies |G| = 0$, a contradiction.

\subsubsection{Groups with \texorpdfstring{$|C(G)|=\frac{|G|}{2}$}{|G|/2} and \texorpdfstring{$\mathcal{O}_2(G)=\frac{|G|}{2}-2$}{|G|/2-2}}\label{2.4.2}
Suppose $|G|=2n$ and $\{x_i\}_{i=1}^{2n}$ are the elements of $G$ such that $x_1 = e$ and $x_i^2= e$ $\forall i \in \{2,3,\cdots,n-1\}$. Then the cyclic subgroup generated by $x_k$ where $k\in \{n,\cdots,2n\}$ must be the same, otherwise $|C(G)|>\frac{|G|}{2}$,
contradict our hypothesis.
Thus $|x_n| \geq n+2\implies |x_n| = 2n \implies G\cong \mathbb Z_{2n}$. As there is only one element of order $2$ in $\mathbb Z_{2n}$. Hence, $n-2 = 1\implies n=3\implies G\cong \mathbb{Z}_{6}$, which is not possible.

\subsubsection{Groups with $|C(G)|=\frac{|G|}{2}$ and $\mathcal{O}_2(G)=\frac{|G|}{2}-3$} \label{2.4.3}
Let $\{x_i\}_{i=1}^{2n}$ be the elements of $G$ such that $x_1 = e$ and $x_i^2= e$ $\forall i \in \{2,3,\cdots,n-2\}$.
Suppose $x_i, 1 \leq i \leq n$ are the generators of the elements of $C(G)$.
Using the fact that,  for each divisor $k$ of $|x_{n}|$ or $|x_{n-1}|$, there are  cyclic subgroups of  order $k$ sitting inside them, we conclude that
  $|x_{n}|$ cannot be the product of two or more distinct odd primes. Therefore, 
 we are left with the following possibilities of $|x_{n-1}|$ and $|x_n|$.
\begin{center}
\begin{tabular}{ |c|c|c|c|c| } 
\hline
 & (a) & (b) & (c) &(d)\\
\hline
$|x_{n}|$ & $2p$ & $p^{2}$  & $p$ & $4$ \\ 
$|x_{n-1}|$ & $p$ & $p$ & $q,p,4$ & $4,8$\\ 
\hline
\end{tabular}
\end{center}
where $p$ and $q$ are odd primes such that $p\neq q$.

If $\langle x_{n-1}\rangle \leq \langle x _{n} \rangle$, then $|x_n| \geq n+3\implies |x_n| = 2n \implies G\cong \mathbb Z_{2n}$. Hence, $n-3 = 1\implies n=4\implies G\cong \mathbb{Z}_{8}$. 
 Hence,  $\langle x_{n-1}\rangle \leq \langle x _{n} \rangle$ or $\langle x_{n}\rangle \leq \langle x _{n-1} \rangle$ holds if and only if $G \cong \mathbb{Z}_{8}$.
 
Case (a) and (b) are not possible because here $\langle x_{n-1}\rangle \leq \langle x _{n} \rangle$. So $G\cong \mathbb{Z}_{8}$ which leads to a contradiction as $p\nmid |G|$.

Case (c):
If $|x_{n}|=p$ and  $|x_{n-1}|=q$,  both are normal in $G$, and  $\langle x_{n} \rangle\langle x_{n-1} \rangle \cong\mathbb{Z}_{pq}$, a contradiction. 
If $|x_{n}|=p$ and  $|x_{n-1}|=p$, then $\phi(p) + \phi(p)=n+2\implies 2p-4=n \implies p\mid 4\implies p=2$, a contradiction. 
When $|x_{n}|=p$ and  $|x_{n-1}|=4$, then $\phi(p) + \phi(4)=n+2\implies p-1=n \implies p\nmid n$, a contradiction. 

Case (d): If $|x_{n}|=|x_{n-1}|=4$, then $\phi{(4)}+\phi{(4)}=n+2\implies n=2$. Thus $G\cong \mathbb{Z}_4$, which is not possible. Finally, if $|x_{n}|=4$ and $|x_{n-1}|=8$, then $\langle x_{n}\rangle \leq \langle x_{n-1}\rangle$ thus $G\cong \mathbb{Z}_{8}$, which is not possible.\\

\noindent \textbf{Proof of Theorem \ref{mainresult_2}.}
\begin{proof}
 (1)  Proof follows from Section \ref{o2(G)=1, |G|/2}.\\
 (2) Proof follows from Section \ref{o2(G)=3, |G|/2}.\\
Proof of (3) and (4) follows from Section \ref{2.4}.
\end{proof}

\subsection{Groups with \texorpdfstring{$|C(G)|=\frac{|G|}{2}-1$}{|C(G)|=|G|/2-1} and \texorpdfstring{$\mathcal{O}_2(G)=1$}{O2(G)=1}}\label{o2G=1, |G|/2-1} 
 For $n \leq 4$ it can be easily checked that $|C(G)| \neq \frac{|G|}{2}-1$ using GAP \cite{GAP}. So, let us assume that $n\geq 5$. Let $\{x_i\}_{i=1}^{2n}$ be the elements of $G$ such that $x_1 = e$ and $x_2^2 = e$. 
 Taking $r=1, k=n-1$,  by our key step given in Section (\ref{keystep}) we have $s=4$. So we are left with four elements say $x_\alpha , x_\beta,x_{\gamma},x_{\delta} \in \{x_{n}, \cdots, x_{2n}\}$ which are not assigned to any class yet. Thus, we have two cases.\\  
  \textbf{Case 1: $x_{\alpha},x_{\beta},x_{\gamma},x_{\delta} \in cl(x_{i})$} for some $i\in\{3,\cdots,n-1\}$.

In this case, $G$ must satisfy the following conditions.

$(A_3)$ There is a unique cyclic subgroup of order two, and there is a unique cyclic subgroup (say $N$) of order $m$ such that $\phi(m) = 6 \implies m\in \{7,9,14,18\}$. By Lemma \ref{normal}, $N\trianglelefteq G$.

$(B_3)$ There are $n-4$ cyclic subgroups say $H_1,\cdots,H_{n-4}$ with orders $m_1,\cdots,m_{n-4}$ such that $\phi(m_i) = 2\implies m_i\in \{3,4,6\} \forall i\in \{1,\cdots,n-4\}$.

In this case, using the proof of Lemma \ref{lemma4}, we can show that there are no groups.\\ 

 \textbf{Case 2:} $x_{\alpha},x_{\beta}\in cl(x_i)$ and $ x_{\gamma},x_{\delta} \in cl(x_{j})$ for some $i,j\in\{3,\cdots,n-1\}$ such that $i\neq j$.

  In this case, $G$ must satisfy the following  conditions.

$(A_4)$ There is a unique cyclic subgroup of order two.

$(B_4)$ There are exactly two cyclic subgroups (say $K_1$ and $K_2$) of order $m_1$ and $m_2$
such that $\phi(m_i) = 4 \implies m_i\in \{5,8,10,12\}$, for $i=1,2$.

$(C_4)$ There are $n-5$ cyclic subgroups say $H_1,\cdots,H_{n-5}$ with orders $m_1,\cdots,m_{n-5}$ 
such that $\phi(m_i) = 2\implies m_i\in \{3,4,6\} $ $\forall i\in \{1,\cdots,n-5\}$.

 Let's call conditions $(A_4)$, $(B_4)$ and $(C_4)$ collectively as property $(P_4)$.

 \begin{lemma}\label{lemma8}
     There is no group $G$ satisfying $(P_4)$ and one of the following conditions:
   \begin{enumerate}
\item $K_1\cong \mathbb{Z}_{10}$ and $K_2 \ncong \mathbb Z_{5}$.
\item  $K_1 \cong \mathbb Z_{5}$ and  $K_2 \ncong \mathbb Z_{10}$.

        \item $K_1,K_2\in \{\mathbb{Z}_{8},\mathbb{Z}_{12}\}$

   \end{enumerate}
 \end{lemma}
\begin{proof}
   (1) and (2) follow using the similar idea given in the proof of Lemma \ref{lemma5}.

    (3) When $K_1\cong \mathbb{Z}_{8}$ and $K_2\cong \mathbb{Z}_{12}$, then result follows from the proof of Lemma \ref{lemma5}.
    
    When $K_1,K_2\cong \mathbb{Z}_8$, following the same idea from the proof of Lemma \ref{lemma6} we get that $|G| = 2^{a}$ where $a\geq 4$.  Also, $a\leq 4$ follows from the ideas in proof of Lemma \ref{Lemma2}. Thus, $a=4$. Using GAP \cite{GAP}, no group of order $16$ satisfies the hypothesis.  
    
    When $K_1,K_2\cong \mathbb{Z}_{12}$, again following the same idea from the proof of Lemma \ref{lemma7} we get that $|G| = 2^{a}3$ where $a\in \{3,4\}$. By GAP \cite{GAP}, there is no group of orders $24$ and $48$ satisfying the hypothesis.
\end{proof}
 
\begin{lemma}\label{lemma9}
    If $G$ satisfies $(P_4)$, then $G$ is a $\{2,5\}$-group. Furthermore $G \cong \mathbb Z_{10}$ or SmallGroup$[20,1]:$ $\mathbb{Z}_{5}\rtimes \mathbb{Z}_{4}$.
\end{lemma}
\begin{proof} 
By Lemma \ref{lemma8} we have $K_1\cong \mathbb{Z}_{10}$ and $K_2 \cong \mathbb Z_{5}$.
    It is also easy to check that $H_{i}\ncong \mathbb{Z}_{3},\mathbb{Z}_6$ $\forall i \in \{1,\cdots, n-5\}$, and hence $|G|=2^{a}5^{b}$ for some $a,b$. We claim that $b=1$, as if $b>1$ then we get a subgroup of order $25$ which would be either $\mathbb{Z}_{25}$ or $\mathbb{Z}_5\times \mathbb{Z}_5$ both contradicting $(P_4)$. Hence, $|G| = 2^a \cdot 5$. Also as $|G|\geq 10$, so $a\geq 1$. Following the reasoning of Lemma \ref{Lemma2}, we have $a\leq 4$. Hence $|G|\in \{10,20,40,80\}$ and using GAP \cite{GAP} we get the result.
\end{proof}

\subsection{Groups with \texorpdfstring{$|C(G)|=\frac{|G|}{2}-1$}{|C(G)|=|G|/2} and \texorpdfstring{$\frac{|G|}{2}-4 \leq$} \texorpdfstring{$\mathcal{O}_2(G)$} $\leq \frac{|G|}{2}-2$}\label{2.6}

\subsubsection{Groups with \texorpdfstring{$|C(G)|=\frac{|G|}{2}-1$}{C(G)=|G|/2-1} and \texorpdfstring{$\mathcal{O}_2(G)=\frac{|G|}{2}-2$}{O2(G)=|G|/2-3}} There is no  such group, which can be proved using similar ideas as in  Section \ref{2.4.1}.

\subsubsection{Groups with \texorpdfstring{$|C(G)|=\frac{|G|}{2}-1$}{C(G)=|G|/2-1} and \texorpdfstring{$\mathcal{O}_2(G)=\frac{|G|}{2}-3$}{O2(G)=|G|/2-3}} Following similar ideas as in Section \ref{2.4.2}, we get $G\cong \mathbb{Z}_{8}$, which is not a solution.

\subsubsection{Groups with \texorpdfstring{$|C(G)|=\frac{|G|}{2}-1$}{C(G)=|G|/2-1} and \texorpdfstring{$\mathcal{O}_2(G)=\frac{|G|}{2}-4$}{O2(G)=|G|/2-4}} 

Let $\{x_i\}_{i=1}^{2n}$ be the elements of $G$ such that $x_1 = e$ and $x_i^2= e$ $\forall i \in \{2,3,\cdots,n-3\}$.
Suppose $x_i, 1 \leq i \leq n-1$ are the generators of the elements of $C(G)$.
Using the fact that  for each divisor $k$ of $|x_{n-1}|$ or $|x_{n-2}|$ there are  cyclic subgroups of  order $k$ sitting in them, we conclude that
  $|x_{n-1}|$ cannot be product of two or more distinct odd primes, and so
 we are left with the following possibilities of $|x_{n-2}|$ and $|x_{n-1}|$.
\begin{center}
\begin{tabular}{ |c|c|c|c|c| } 
\hline
 & (a) & (b) & (c) &(d)\\
\hline
$|x_{n-2}|$ & $2p$ & $p^{2}$  & $p$ & $4$ \\ 
$|x_{n-1}|$ & $p$ & $p$ & $q,p,4$ & $4,8$\\ 
\hline
\end{tabular}
\end{center}
where $p$ and $q$ are odd primes such that $p\neq q$.

If $\langle x_{n-1}\rangle \leq \langle x _{n-2} \rangle$, then $|x_{n-2}| \geq n+4\implies |x_{n-2}| = 2n \implies G\cong \mathbb Z_{2n}$. Hence, $n-4 = 1\implies n=5\implies G\cong \mathbb{Z}_{10}$. 
 Hence,  $\langle x_{n-1}\rangle \leq \langle x _{n-2} \rangle$ or  $\langle x_{n-2}\rangle \leq \langle x _{n-1} \rangle$ holds if and only if $G \cong \mathbb{Z}_{10}$.
 
Case (a): $\phi(2p)+\phi(p)=n+3\implies 2p-5=n\implies p=n=5\implies G\cong \mathbb{Z}_{10}$ which is a solution.

Case (b):  Here  $\langle x_{n-1}\rangle \leq \langle x _{n-2} \rangle$ holds thus $G\cong \mathbb{Z}_{10}$, but it doesn't have cyclic subgroup with order $p^{2}$ for any $p$.

Case (c):
If $|x_{n-2}|=p$ and  $|x_{n-1}|=q$ or $4$, then similar ideas as in Section \ref{2.4.3} will give a contradiction. If $|x_{n-1}|=|x_{n-2}|=p$, then $\phi(p) + \phi(p)=n+3\implies 2p-5=n\implies p=n=5$ and $|G|=10$ but both groups of order $10$ do not have two cyclic subgroups of order $5$. 

Case (d): When $|x_{n-1}|=|x_{n-2}|=4$, we have $\phi{(4)}+\phi{(4)}=n+3\implies n=1$. Thus $G\cong \mathbb{Z}_2$, which is not a solution. And finally when $|x_{n-2}|=4$ and $|x_{n-1}|=8$, then $\langle x_{n-2}\rangle \leq \langle x_{n-1}\rangle$ thus $G\cong \mathbb{Z}_{10}$, which doesn't have cyclic subgroup of order $4$ and $8$. 
\\

\noindent\textbf{Proof of Theorem \ref{mainresult_3}}
\begin{proof}
 (1) Proof follows from Section \ref{o2G=1, |G|/2-1}. 
 
 Proof of (2) and (3) follows from  Section \ref{2.6}.
\end{proof}


\noindent\textbf{Proof of Theorem \ref{mainresult_4}}
\begin{proof}
 (1) Follows from Theorem  \ref{mainresult} when $k=0$.

(2), (3), (4) and (5) can be proved using similar ideas as used in proof of Theorem \ref{mainresult_2} and Theorem \ref{mainresult_3}.
\end{proof}

\noindent \textbf{Proof of Theorem \ref{mainresult_5}}.
\begin{proof}
 In this case, the group must be of odd order. Let $\{x_i\}_{i=1}^{2n+1}$ be the elements of $G$ such that $x_1 =e$.
 
 (1)  It can be easily checked that there are no solutions for $n=0,1$. So assume that $n>1$. Then  $cl(x_k)\geq 2 $ $\forall k\in \{2,\cdots, n\}$. Therefore, repeating the same process as in the proof of Theorem \ref{mainresult_2}, we get the following conditions on $G$.

$(A_5)$ There is a unique cyclic subgroup (say $N$) of order $m$ such that $\phi(m) = 4 \implies m\in \{5,8,10,12\}$. By Lemma \ref{normal}, $N\trianglelefteq G$.

$(B_5)$ There are $n-2$ cyclic subgroups say $H_1,\cdots,H_{n-2}$ with orders $m_1,\cdots,m_{n-2}$ such that $\phi(m_i) = 2\implies m_i\in \{3,4,6\} \forall i\in \{1,\cdots,n-2\}$.

Let's call conditions $(A_5)$ and $(B_5)$ collectively as property $(P_5)$.

Since, the group is of odd order, therefore in $(A_5)$, $m=8, 10, 12$ are not possible. For the same reason, in $(B_5)$, $m_i = 4,6$ are not possible. So the only possibility is that $N\cong \mathbb Z_5$ and $H_i\cong \mathbb Z_3$ $\forall i$. Now if $n>2$ then such a $H_i$ exists and we get $NH_i\leq G$ where $NH_i\cong \mathbb Z_{15}$ contradicting $(P_5)$. So, the only possibility is $n= 2$. For $n=2$, $G\cong \mathbb Z_5$ and it satisfies $|C(G)| = \frac{|G|-1}{2}$.\\

(2) Let $|G| = 2n+1$ where $n\geq 0$. Then $G$ have $n$ cyclic subgroups say $H_1,\cdots,H_{n}$ of orders $m_1,\cdots,m_{n}$ such that $\phi(m_i) = 2\implies m_i\in \{3,4,6\} \forall i\in \{1,\cdots,n\}$.  Since $G$ is of odd order, $m_i\neq 4,6$. So every non-trivial element of $G$ must be of order $3$. Hence, $G\cong\{e\}$ or $G$ is a group with $\exp(G)=3$.\\

(3) For $r\geq 1$, repeating the same process as in the proof of (1), this time we are left with $2r$ non-trivial elements, each of them having order $m$ such that $\phi(m)=1$, hence each of them is an involution which contradicts the fact that group is odd.
\end{proof}

\noindent \textbf{Proof of Lemma \ref{bound}}.
\begin{proof}
In this case, by our key step given in \ref{keystep}, we have $s=\mathcal{O}_{2}(G)+1-2m$. Hence, the largest number of generators  a cyclic subgroup of $G$ can have $\mathcal{O}_2(G)+3-2m$ elements. For any prime divisor $p$ of  $|G|$, we have $\phi(p)=p-1 \leq \mathcal{O}_2(G)+3-2m\implies p\leq \mathcal{O}_2(G)+4-2m$. 
\end{proof}

Using the similar ideas used in the proof of the main results, we have the following. 
\begin{lemma}
     Let $G$ be a group of order $2n$ and $|C(G)|=n+r$ for a fixed positive integer $r$.
     \begin{enumerate}
         \item If  $1\leq  r\leq n $ and $\mathcal{O}_{2}(G)=2r-1$,  then $G$ is a $\{2\}$-group or $\{2,3\}$-group. 
         \item If $1\leq r\leq n -3 $ and $\mathcal{O}_{2}(G)=2r+1$, then $G$ is a $\{2\}$-group or $\{2,3\}$-group or $\{2,5\}$-group.
     \end{enumerate}
      \end{lemma}

\vspace{1cm}

\noindent\textbf{Acknowledgements:}
The first named author acknowledges the research support of the SERB-SRG, Department of Science and Technology (Grant no. SRG/2023/000894), Govt. of India. The third named author thank the National Institute of Science Education and Research, Bhubaneswar, for providing a summer intern fellowship. The authors thank the National Institute of Science Education and Research, Bhubaneswar, for providing an excellent research facility.

\end{document}